\documentclass{amsart}
\usepackage{xypic}
\usepackage{graphicx, amsmath, amssymb, amsthm}
\usepackage{tikz-cd}
\usepackage{hyperref}

\newcommand{\mycases}[1]{\left\{\begin{array}{ll}#1\end{array}\right.}
\newcommand{\Z}{{\mathbb  Z}}

\newcommand{\R}{{\mathbb R}}
\newcommand{\F}{{\mathbb F}}
\newcommand{\MU}{MU}
\newcommand{\BP}{BP}
\newcommand{\BPG}{\BP^{(\!(G)\!)}}
\newcommand{\BPGm}{\BPG\langle m\rangle}
\newcommand{\MUR}{\MU_{\R}}
\newcommand{\MUG}{\MU^{(\!(G)\!)}}
\newcommand{\BPR}{BP_{\R}}
\newcommand{\BUR}{BU_{\R}}
\newcommand{\Xin}{\Xi_{n}}

\DeclareMathOperator{\Ind}{Ind}
\DeclareMathOperator{\Res}{Res}
\DeclareMathOperator{\ind}{Ind}

\DeclareMathOperator{\Ext}{Ext}
\DeclareMathOperator{\Tor}{Tor}

\DeclareMathOperator{\Map}{Map}
\DeclareMathOperator{\Sym}{Sym}

\DeclareMathOperator{\Stab}{Stab}

\DeclareMathOperator{\RO}{RO}

\newcommand{\br}{\bar{r}}

\newcommand{\m}[1]{{\protect\underline{#1}}}
\newcommand{\mZ}{\m{\Z}}

\newcommand{\mR}{\m{R}}

\newcommand{\mF}{\m{\F}}
\newcommand{\mpi}{\m{\pi}}

\newcommand{\mA}{\m{A}}

\newcommand{\mstar}{\m{\star}}
\newcommand{\mRO}{\m{\RO}}

\mathchardef\mhyphen=45

\newcommand{\EM}{Eilenberg-Mac~Lane}
\newcommand{\smashover}[1]{\underset{#1}{\otimes}}
\newcommand{\otimesover}[1]{\underset{#1}{\otimes}}
\newcommand{\Boxover}[1]{\underset{#1}{\Box}}

\numberwithin{equation}{section}

\newtheorem{theorem}{Theorem}[section]
\newtheorem{lemma}[theorem]{Lemma}
\newtheorem{corollary}[theorem]{Corollary}

\newtheorem{proposition}[theorem]{Proposition}
\newtheorem{conjecture}[theorem]{Conjecture}
\newtheorem*{theorem*}{Theorem}
\newtheorem*{proposition*}{Proposition}

\theoremstyle{remark}
\newtheorem{remark}[theorem]{Remark}

\newtheorem{notation}[theorem]{Notation}

\theoremstyle{definition}

\newtheorem{definition}[theorem]{Definition}

\title[Hyperreal Steenrod]{On the hyperreal dual Steenrod algebra}

\author{Michael A.~Hill}
\address{University of Minnesota, Minneapolis, MN 55455}
\email{\tt{mahill@umn.edu}}

\author{Michael J.~Hopkins}
\address{Department of Mathematics, Harvard University, Cambridge, MA 02138}
\email{{\tt{mjh@math.harvard.edu}}}

\begin{document}

\begin{abstract}
We compute the dual Steenrod algebra for Bredon homology with constant coefficients \(\underline{\mathbb Z}\) and \(\underline{\mathbb Z}/2\) in the category of modules over \(\MU^{((G))}\), the norm to \(G=C_{2^n}\) of \(\MU_{\mathbb R}\). Using this and an equivariant version of the Greenlees--Serre spectral sequence, we give a spectral sequence computing the \(RO\)-graded homotopy of the Eilenberg--Mac Lane spectrum \(H\underline{\mathbb F}_2\otimes H\underline{\mathbb Z}\).
\end{abstract}

\maketitle

\section{Introduction}

The Reduction Theorem of \cite{HHR} was a key ingredient in determining the slices of the \(G=C_{2^n}\)-spectrum
\[
    \Xin =\MUG=N_{C_{2}}^{C_{2^{n}}}\MUR.
\]
It did so by giving a kind of presentation of \(H\mZ\), the {\EM} spectrum for the constant Mackey functor \(\mZ\), as a module over \(\Xin\), as we recall in Section~\ref{sec:TwistedMonoidRings}. This actually gives us a formula for the smash product of an arbitrary \(G\)-spectrum with \(H\mZ\).

\begin{theorem}
For any \(G\)-spectrum \(E\), there are natural weak equivalences
\[
H\mZ\otimes E\simeq \Xin\smashover{A} E,
\]
where 
\[
A=\mathbb S^0[G\cdot \bar{r}_1,\dots]:=\bigotimes_{i\geq 1}N_{C_2}^{G}\mathbb S^0[S^{i\rho_2}]
\]
is the twisted monoid algebra specified by a choice of generators of \(\pi_{\ast\rho_2}\Xin\) and acts on \(E\) via the augmentation to \(\mathbb S^0\).
\end{theorem}

This formula is more useful when \(E\) is a structured ring spectrum, and it is most helpful when \(E\) in in fact \(\Xin\)-oriented. If \(E\) is an \(E_\infty\) ring spectrum, then we have a natural equivalence of \(G\)-spectra
\[
    \Xin\otimesover{A} E\simeq (E\otimes\Xin)\otimesover{E\otimes A}E.
\]
When \(E\) is \(\Xin\)-orientable, then we can recast the algebra map here. The homotopy of \(E\otimes A\) is free in an \(\mRO\)-graded sense, and the same is true for \(E\otimes\Xin\), via the Thom isomorphism (See Corollary~\ref{cor:Homotopy of EXin}). The map on homotopy from \(E\otimes A\) to \(E\otimes\Xin\) is then a kind of \(E\)-based Hurewicz image.

When \(E=H\mZ\) or \(H\mF_2\), then we can further recast this purely in equivariant algebra. The \(\mRO\)-graded homotopy of \(H\mZ\otimes A\) is the free graded Tambara \(\pi_{\mstar}H\mZ\)-algebra on \(C_2\)-equivariant classes \(\br_1,\br_2,\dots\). The classes \(\br_i\) are homotopy elements in \(\pi_{\ast\rho_2}^{C_2}(\Xin)\), and so we are considering their Hurewicz images in 
\[
    \pi_{\ast\rho_2}^{C_2} \big(H\mZ\otimes \Xin\big)\cong \Z[\bar{m}_1,\gamma\bar{m}_1,\dots].
\]
This is essentially the way we understand the classes \(\bar{r}_i\). In practice, we find that we have to compute an \(\mRO\)-graded \(\Tor\) group for a complicated map of equivariant algebras. Unpacking this, we will show the following.
\begin{theorem}
    We have an isomorphism of \(\pi_{\mstar}H\mF_2\)-modules
    \[
        {H\mF_2}_{\mstar}H\mZ\cong \Tor^{-\mstar}_{H\mF_{2\mstar}[G\cdot\br_1,\dots]}\big(H\mF_{2\mstar}[G\cdot\bar{m}_1,\dots],H\mF_{2\mstar}\big).
    \]
\end{theorem}

To simplify the computation, we introduce an intermediate object, the \(\Xin\)-relative dual Steenrod algebra. Since \(\Xin \) is a \(G\)-commutative ring spectrum, there is a \(G\)-symmetric monoidal category of modules \cite{BHModules}. Given any commutative \(\Xin \)-algebras \(E\) and \(E'\), we can form the commutative \(\Xin \)-ring spectrum
\[
E\smashover{\Xin} E'.
\]
When \(E=E'\), then this plays the role of the ring of co\"operations for the spectral Hopf algebroid recording descent from \(E\)-modules to \(\Xin\)-modules. In particular, assuming nice flatness properties, we have a simpler Adams spectral sequence computing the homotopy of a \(\Xin\)-module, a la Baker--Lazarev \cite{BakerLazarev}. 

Working over \(\Xin\) makes many computations much simpler. We can give a complete description of the integral and mod \(2\) \(\Xin\)-relative dual Steenrod algebras. We work out the additive story in Section~\ref{sec:Additive}, proving the following results.

\begin{theorem}\label{thm:IntegralAdditiveRelativeSteenrod}
For any \(\mZ\)-algebra \(\mR\), we have a weak equivalence of \(H\mR\)-modules
\[
H\mR\smashover{\Xin }H\mZ\simeq H\mR\otimes \Sigma^\infty_+\left(\prod_{i\geq 1}\Map^{C_{2}}\big(G,S^{i\rho_{2}+1}\big)\right).
\]
\end{theorem}

This completely describes the additive structure. Since products and coinduction split after applying the infinite complete suspension, the right-hand side is \(H\mR\)-free, splitting as a wedge of \(H\mR\)-modules of the form
\[
H\mR\otimes (G_+\smashover{H} S^{V})
\]
for various subgroups \(H\) and \(H\)-representations \(V\). We observe here that we made no assumptions about \(\mR\); it can be as pathological as desired.

Specializing to the case \(\mR=\mF_2\), we can also understand the relative smash square of \(H\mF_2\).

\begin{theorem}\label{thm:modTwoAdditiveRelativeSteenrod}
We have a weak equivalence
\[
H\m{\F}_{2}\smashover{\Xin} H\m{\F}_{2}\simeq 
H\mF_2\otimes\Sigma^{\infty}_+\left(S^1\times\prod_{i\in\mathbb N}\Map^{C_{2}}\big(G,S^{i\rho_{2}+1}\big)\right).
\]
\end{theorem}

In Section~\ref{sec:Multiplicative}, we study the multiplicative structure. If \(\mR\) is a commutative Green \(\mZ\)-algebra, then
\[
H\mR\smashover{\Xin} H\mZ
\]
is an \(E_\infty\)-monoid in \(H\mZ\)-modules. If \(\mR\) is in fact a Tambara \(\mZ\)-algebra, then the relative tensor product is a \(G\)-commutative \(H\mZ\)-algebra, and so it has natural norm maps on its homotopy. Surprisingly, this is the easiest case, since our approach and splitting preserve certain kinds of norm structure automatically. We can read these out of our computation and then use properties of \(G\)-spectra to resolve other multiplicative extensions.

Finally, we return to the global case. In Section~\ref{sec:AbsoluteCase}, we produce an equivariant refinement of the Greenlees--Serre spectral sequence \cite{GreenleesSS}. Given a pushout square of commutative ring spectra
\[
\begin{tikzcd}
    {A}
        \ar[r]
        \ar[d]
        &
    {R}
        \ar[d]
        \\
    {B}
        \ar[r]
        &
    {Q,}
\end{tikzcd}
\]
the Greenlees--Serre spectral sequence computes the homotopy groups of \(R\) (the ``total space'') out of the homotopy groups of \(A\) (the ``fiber''), \(B\), and \(Q\) (the ``base''). Greenlees' version is built out of the Postnikov tower for \(A\) in the case \(B\) is the zeroth Postnikov section. Ours is built out of the slice tower for \(A\) in the case \(B\) is the zeroth slice section.

\begin{theorem}
    Let \(f\colon A\to R\) be a map of connective equivariant commutative ring spectra, let \(P^0\) be the zero slice section functor, and consider the pushout square of equivariant commutative ring spectra
    \[
    \begin{tikzcd}
        {A}
            \ar[r, "f"]
            \ar[d]
            &
        {R}
            \ar[d]
            \\
        {P^0A}
            \ar[r]
            &
        {Q.}
    \end{tikzcd}
    \]
    We have a spectral sequence of graded Tambara functors with \(E_2\)-term
    \[
    E_2^{s,\mstar}=\pi_{\mstar-s}\big(Q\smashover{P^0A} P_{\dim\mstar}^{\dim\mstar}(A)\big),
    \]
    where \(P_{\dim\mstar}^{\dim\mstar}(A)\) is the \(\dim\mstar\)-slice \(A\). The differentials are Adams type. This converges strongly to the homotopy groups of \(R\).
\end{theorem}

In general, the \(E_2\)-term will be tricky to compute, given the relative tensor that appears.  There is a Lewis--Mandell K\"unneth spectral sequence computing the homotopy of the relative smash product \cite{LewisMandell}. In the case that the slices of \(A\) are free \(P^0(A)\)-modules, we can easily identify the \(E_2\)-term. This is the case for \(H\mF_2\otimes \Xin\).

\begin{corollary}\label{cor:RelativeToAbsolute}
    There is a spectral sequence of graded Tambara functors with \(E_2\)-term
    \[
        \pi_{\mstar}\big(H\mF_2\smashover{\Xin} H\mZ\big)[G\cdot \bar{m}_1,\dots]\Rightarrow \pi_{\mstar}\big(H\mF_2\otimes H\mZ\big).
    \]
\end{corollary}

Again, this case is actually purely algebraic. The differentials are built out of remembering that the classes \(\bar{r}_i\) are algebraic combinations of the classes \(\bar{m}_i\) and the algebra generators of \(\pi_{\mstar} \big(H\mF_2\smashover{\Xin} H\mZ\big)\) are a kind of formal null-homotopy of the classes \(\bar{r}_i\). In other words, we can reinterpret Corollary~\ref{cor:RelativeToAbsolute} as a kind of Kozsul resolution for the homology of \(H\mZ\) with \(\mF_2\)-coefficients. We sketch this complete computation for \(G=C_2\), rebuilding the work of Hu--Kriz \cite{HuKriz}.

\subsection*{Notation}
We work in the category of genuine \(G\) spectra, viewed as a \(G\)-symmetric monoidal stable \(\infty\)-category. 

The letter \(G\) will almost exclusively refer to the cyclic group \(C_{2^n}\), and letters \(K\) and \(H\) will often denote subgroups. 

Let \(\sigma\) denote the sign representation. Let \(\lambda\) denote the underlying real representation of the complex representation of \(G\) which takes a generator to multiplication by \(e^{2\pi i/2^n}\). Let \(\rho_H\) denote the regular representation of a group \(H\), and let \(\rho_2\) denote the regular representation of \(C_2\).

Given a representation \(V\) of \(G\), let
\[
    a_V\colon S^0\to S^V
\]
denote the inclusion of \(\{0,\infty\}\) into \(S^V\).

Let \(\mathbb M_2\) denote the \(RO(C_2)\)-graded homology of a point. 

The wildcard for integer gradings will be an asterisk. For \(RO(G)\)-gradings, we will use a star: \(\star\). Finally, we will also work often with a more general grading, allowing not only (virtual) representations of \(G\), but also induced up (virtual) representations of \(H\) (See \cite{AngeltveitBohmann, HillFreeness}). The corresponding ``wildcard'' in this \(\mRO\)-grading is \(\m{\star}\).

Finally, several computations will take place in various K\"unneth spectral sequences and use the homology suspension. To avoid confusion with the sign representation, we will use \(\varsigma x\) for the homology suspension of an element \(x\) in the homotopy of a ring spectrum.

\subsection*{Acknowledgements}
The authors thank Agn\`{e}s Beaudry, Prasit Bhattacharya, Tyler Lawson, and Danny XiaoLin Shi for their close reading of early drafts of this paper. Part of the work was completed while the first author was visiting the Isaac Newton Institute for Mathematical Sciences, Cambridge, during the program Equivariant Homotopy Theory in Context and supported by EPSRC grant no EP/Z000580/1.

\section{Twisted monoid rings and the Reduction Theorem}\label{sec:TwistedMonoidRings}

The proof of Theorem~\ref{thm:IntegralAdditiveRelativeSteenrod} uses the Reduction Theorem of \cite{HHR} in its guise as a presentation of \(H\mZ\) as a \(\Xin\)-module. We recall and recast some of the necessary pieces here. 

\subsection{Free associative algebras}
Recall that the infinite suspension functor is a strong symmetric monoidal functor, taking the Cartesian product of \(G\)-CW complexes to the smash product of genuine \(G\)-spectra. In particular, we have a variant of the Snaith splitting.

\begin{definition}
If \(X\) is a \(G\)-CW spectrum, then let
\[
\mathbb S^{0}[X]=\bigoplus_{i\in\mathbb N}X^{\otimes i}
\]
be the free associative algebra on \(X\).
\end{definition}

\begin{lemma}\label{lem:James}
The free associative algebra on a pointed, simply connected \(G\)-space \(X\) is given by
\[
\mathbb S^{0}[\Sigma^{\infty}_{+}X]=\Sigma^{\infty}_{+} \Omega\Sigma X.
\]
\end{lemma}

Since the infinite suspension is a strong symmetric monoidal functor from spaces with the Cartesian product to spectra, we also have a related formula for smash products of free associative algebras on suspension spectra of  spaces.

\begin{lemma}\label{lem:SmashofFrees}
The smash product of the free associative algebras on \(X\) and on \(Y\) is given by
\[
\mathbb S^{0}[\Sigma^{\infty}_{+}X]\otimes\mathbb S^{0}[\Sigma^{\infty}_{+}Y]\simeq\Sigma^{\infty}_{+}\Omega\big(\Sigma X \times\Sigma Y\big),
\]
as associative algebras.
\end{lemma}

This approach gives us a way to interpret the smash products over a free associative algebra (or the smash products thereof).

\begin{lemma}\label{lem:SmashOverFree}
Let \(X\) be a path connected \(G\)-CW complex. We have a natural weak equivalence
\[
\mathbb S^{0}\smashover{\Sigma^{\infty}_{+}\Omega X}\mathbb S^{0}\simeq \Sigma^{\infty}_{+} X.
\]
\end{lemma}

\begin{corollary}\label{cor:SmashOverFrees}
For any spaces \(X\) and \(Y\), we have
\[
\mathbb S^{0}\smashover{\mathbb S^{0}[\Sigma^{\infty}_{+}X]\otimes\mathbb S^{0}[\Sigma^{\infty}_{+}Y]}\mathbb S^{0}\simeq \Sigma^{\infty}_{+} \big(\Sigma X\times\Sigma Y\big),
\]
and similar for arbitary smash products.
\end{corollary}

The infinite suspension functor is actually a \(G\)-symmetric monoidal functor, taking coinduction in \(G\)-spaces to the norm.
\begin{proposition}\label{prop:SuspensionGMonoidal}
    For any \(H\)-space \(X\), we have a natural weak equivalence
    \[
        N_H^G\big(\Sigma^{\infty}_{+}X\big)\simeq \Sigma^{\infty}_+\Map^H(G,X).
    \]
\end{proposition}
This gives a twisted version of Lemma~\ref{lem:SmashofFrees}.

\begin{lemma}\label{lem:NormOfFreeAssoc}
The norm of the free associative algebra on a pointed \(H\)-space \(X\) is given by
\[
N_{H}^{G}\big(\mathbb S^{0}[\Sigma^{\infty}_{+}X]\big)\simeq \Sigma^{\infty}_{+}\Map^{H}(G,\Omega\Sigma X)\simeq\Sigma^{\infty}_{+}\Omega\Map^{H}(G,\Sigma X).
\]
\end{lemma}
\begin{proof}
The first equivalence is the composite of Lemma~\ref{lem:James} and Proposition~\ref{prop:SuspensionGMonoidal}. The weak equivalence is just commuting the two right adjoints in \(G\)-spaces. 
\end{proof}

\subsection{Twisted monoid rings and \texorpdfstring{\(\Xin\)}{Xin}}

Given a \(G\)-commutative ring spectrum \(R\) and an element in the \(H\)-equivariant homotopy
\[
    \bar{x}\in \pi_V^H(R)
\]
for some (virtual) \(H\)-representation \(V\), we can build a map of associative algebras:
\[
    \mathbb S^0[S^V]\to i_H^\ast R.
\]
Since \(R\) is a \(G\)-commutative monoid, we can then norm this up to \(G\) and compose with the counit of the norm-forget adjunction, getting a map of associative rings
\[
    N_H^G\big(\mathbb S^0[S^V]\big)\to N_H^G i_H^\ast R\to R.
\]
Given multiple elements, we can tensor these algebras together. Our analysis hinges on the fact that we can understand the source algebra as the suspension spectrum of a space, and that for \(R=\Xin\), we have well-behaved generators.

To state this, let \(\mathbb Z_{-}\) denote the integral sign representation of \(C_2\).

\begin{theorem}[{\cite[Proposition 5.27, Lemma 5.33]{HHR}}]
There are classes
\[
\bar{r}_{i}\colon S^{i\rho_{2}}\to i_{C_{2}}^{\ast}\Xin
\]
such that the induced map of commutative rings
\[
\bigotimes_{i\geq 1} \Sym\big(\Ind_{C_2}^{G} (\Z_{-})^{\otimes i}\big) \xrightarrow{\otimes \ind\bar{r}_i} \mpi_{\ast\rho_2} (\Xin)(C_2/C_2)
\]
is an isomorphism.
\end{theorem}

\begin{theorem}[{\cite[Reduction Theorem]{HHR}}]
Thom reduction gives a weak equivalence of \(\Xin\)-modules
\[
\Xin\smashover{A}\mathbb S^{0}\simeq H\mZ.
\]
\end{theorem}

\section{\texorpdfstring{\(\mZ\)}{Z}-homology via \texorpdfstring{\(\Xin\)}{Xin}}
\subsection{General Formula}
The Reduction Theorem immediately gives a way to describe \(H\mZ\) smashed with any \(G\)-spectrum. 

\begin{proposition}
    For any \(G\)-spectrum \(E\), we have an equivalence of \(\Xin\)-modules
    \[
    H\mZ\otimes E\simeq \Xin\smashover{A} E,
    \]
    where \(E\) is viewed as an \(A\)-module via the augmentation map \(A\to \mathbb S^0\).
\end{proposition}
\begin{proof}
    This is the shuffle equivalence
    \[
    H\mZ\otimes E\simeq (\Xin\smashover{A} \mathbb S^0)\otimes E\simeq \Xin\smashover{A} E.\qedhere    
    \]
\end{proof}

If \(E\) is an \(E_\infty\) monoid or a \(G\)-commutative monoid, then this equivalence is one of \(\Xin\otimes E\)-modules. In this case, we can use the free-forget adjunction for \(E\)-modules now to rewrite this.

\begin{corollary}
    If \(E\) is an \(E_\infty\) monoid, then we have an equivalence of \(\Xin\otimes E\)-modules
    \[
        H\mZ\otimes E\simeq (E\otimes \Xin)\smashover{E\otimes A} E.
    \]
\end{corollary}

\begin{remark}
    Since \(A\) is a wedge of spectra of the form \(G_+\otimesover{H}S^V\), the associative ring spectrum \(E\otimes A\) is always a free \(E\)-module.
\end{remark}

\begin{remark}
    This applies even in the case \(E=H\mA\), the {\EM} spectrum for the Burnside Mackey functor and hence by base-change, for any {\EM} spectrum.
\end{remark}

\subsection{\texorpdfstring{\(\Xin\)}{Xin}-orientable case}
When \(E\) is \(\Xin\)-orientable, the Thom isomorphism allows us to rewrite the first smash factor.
\begin{proposition}\label{prop:Thom Diagonal}
Given a \(\Xin\)-orientable ring spectrum \(E\), we have an equivalence of left \(E\)-modules
\[
    E\otimes \Xin\simeq E\otimes\Sigma^{\infty}_{+} \Map^{C_2}(G,\BUR).
\]
\end{proposition}
This has free \(E\)-homology, essentially for the same reasons as the classical cases. A \(\Xin\)-orientation gives a Real orientation of \(i_{C_2}^{\ast}E\) by restriction and the unit of the norm-forget adjunction:
\[
    \MUR\to i_{C_2}^{\ast}\Xin.
\] 
This allows us to easily identity the \(i_{C_2}^\ast E\)-homology of \(\BUR\) and the \(E\)-homology of the coinduced \(\BUR\).

\begin{theorem}[{\cite[]{HillFreeness}}]\label{thm:Splitting Oriented Homology}
    There are classes 
    \[
        \bar{m}_i\in \pi_{i\rho_2}^{C_2}\big(\Xin\otimes \Map^{C_2}(G,\BUR)\big)
    \]
    such that the induced map
    \[
        \Xin\otimes \Big(\bigotimes_{i\geq 1} N_{C_2}^{G}\mathbb S^0[S^{i\rho_2}]\Big)\to \Xin\otimes \Map^{C_2}(G,\BUR)_+
    \]
    is an equivalence of associative \(\Xin\)-algebras.
\end{theorem}

Base-changing along the map \(\Xin\to E\) and using the Thom isomorphism allows us to deduce the general case.

\begin{corollary}\label{cor:Homotopy of EXin}
    For any \(\Xin\)-orientable \(G\)-spectrum \(E\), the associative \(E\)-algebra map
    \[
        E\otimes \Big(\bigotimes_{i\geq 1} N_{C_2}^{G} \mathbb S^0[S^{i\rho_2}]\Big)\to E\otimes\Xin
    \]
    is an equivalence.
\end{corollary}

Note that the associative \(E\)-algebra described is actually just \(E\otimes A\). The induced \(A\)-algebra structure on \(E\otimes\Xin\), however, is not the one coming from \(\Xin\) itself. Instead, the self-map in the homotopy category induced by
\[
    E\otimes A\xrightarrow{E\otimes \iota} E\otimes \Xin\xleftarrow{\simeq} E\otimes \Big(\bigotimes_{i\geq 1} N_{C_2}^{G} \mathbb S^0[S^{i\rho_2}]\Big)
\]
is not the identity, however. This \(E\)-algebra map is the one determined by the \(C_2\)-equivariant maps
\[
S^{i\rho_2}\to i_{C_2}^{\ast}\Xin\to i_{C_2}^{\ast}(E\otimes \Xin).
\]
In other words, the map from \(E\otimes A\) is picking out the \(E\)-based Hurewicz image in \(E\otimes\Xin\).

\begin{notation}
    When we want to stress that the algebra structure is this Hurewicz one, we will denote \(A\) by \(A^h_n\).
\end{notation}

\begin{theorem}\label{thm: Tor SS for Homology}
    Let \(E\) be an \(E_\infty\)-monoid, and assume \(E\) is \(\Xin\)-oriented. Then we have an Adams-graded spectral sequence with \(E_2\)-term
    \[
        E_2^{s,V}=\Tor_{E_{\mstar}[G\cdot\bar{r}_1,\dots]}^{-s,V}\big(E_{\mstar}[G\cdot \bar{m}_1,G\cdot\bar{m}_2,\dots] ,{E_\mstar}\big)
    \]
    converging to \(\pi_{V-s}\big(H\mZ\otimes E\big)\).
\end{theorem}
\begin{proof}
    This is the K\"unneth spectral sequence associated to writing \(H\mZ\otimes E\) as a balanced smash product:
    \[
        H\mZ\otimes E\simeq (E\otimes\Xin)\smashover{E\otimes A} E.
    \]
    Corollary~\ref{cor:Homotopy of EXin} identifies the homotopy of the first smash factor. The homotopy of the ground ring is the definition of the ring over which we are taking \(\Tor\). The result follows.
\end{proof}

In the case that \(E=H\mF_2\), then we are working in the category of \(H\mF_2\)-modules which is the derived category of \(\mF_2\)-modules \cite{SchwedeShipley}. In this case, the {\EM} functor is [derived] strong symmetric monoidal, and the K\"unneth spectral sequence here collapses. For \(n\geq 2\), this is a \(\Tor\)-computation in an abelian category with which we have little familiarity. For \(n=1\), however, we deduce the additive dual Steenrod algebra with easy, since the K\"unneth spectral sequence is just a form of the Lewis--Mandell \(RO(C_2)\)-graded one.

\begin{corollary}\label{cor: Hu--Kriz Result}
    We have an isomorphism of \(RO(C_2)\)-graded modules
    \[
        \mathcal A_{\star}\cong \mathbb M_2[\bar{x}_1,\dots]\otimesover{\mathbb M_2} E_{\mathbb M_2}(\tau_0,\tau_1,\dots),
    \]
    where \(|\bar{\xi}_n|=|\bar{v}_n|=(2^{n}-1)\rho_2\), and \(|\tau_m|=|\bar{v}_m|+1=2^{n-1}+(2^{n-1}-1)\varsigma.\)
\end{corollary}
\begin{proof}
    The Hurewicz map takes the form
    \[
        \br_i\mapsto\mycases{
            \bar{m}_i & i\neq 2^n-1 \forall n \\
            0 & \text{ otherwise.}
        }
    \]
    This gives the desired result, where the elements \(\bar{\xi}_n\) are represented by \(\bar{m}_{2^{n}-1}\) and where \(\bar{\tau}_m\) is represented by the homology suspension of \(\bar{r}_{2^m-1}\).
\end{proof}

\section{The relative dual Steenrod algebra: additive structure}\label{sec:Additive}
\subsection{The relative smash product with \texorpdfstring{\(H\mZ\)}{HZ}}

Using the presentation of \(H\mZ\) as a \(\Xin\)-module, we can also compute the various relative smash products of \(G\)-spectra with \(H\mZ\).

\begin{proposition}\label{prop:GeneralRelativeCase}
    Given any \(\Xin\)-module \(E\), we have equivalences
    \[
        E\smashover{\Xin} H\mZ\simeq E\smashover{A} \mathbb S^0,
    \]
    where \(E\) is an \(A\)-module via pulling back along \(A\to\Xin\).
\end{proposition}
\begin{proof}
    This is the change-of-rings isomorphism. We have natural equivalences
    \[
        E\otimesover{\Xin}H\mZ\simeq E\otimesover{\Xin}(\Xin\otimesover{A}\mathbb S^0)\simeq E\otimesover{A}\mathbb S^0.\qedhere
    \]
\end{proof}

In the case that \(E=H\mZ\) itself, then we can more easily further rewrite this, since the \(A\)-module structure is necessarily trivial.

\begin{theorem}\label{thm:Relative Homology as Spectrum}
    There is an equivalence of \(H\mZ\)-modules
    \[
        H\mZ\smashover{\Xin} H\mZ\simeq
        H\mZ\otimes\Sigma^{\infty}_{+} \left(\prod_{i=1}^{\infty} \Map^{C_2}(G,S^{i\rho_2+1})\right).
    \]
\end{theorem}

\begin{proof}
Proposition~\ref{prop:GeneralRelativeCase} gives an equivalence of \(H\mZ\)-modules
\[
H\mZ\smashover{\Xin}H\mZ\simeq B(H\mZ,A,\mathbb S^0).
\]
The map \(\Xin\to H\mZ\) is a ring map, and hence the module structure is induced from the associative ring map
\[
A\to H\mZ.
\]
Since \(H\mZ\) is a zero-slice and since all of the summands of \(A\) beyond the copy of \(\mathbb S^{0}\) are slice positive, any map
\[
A\to H\mZ
\]
factors as
\[
A\to \mathbb S^{0}\to P^0(\mathbb S^0)\to H\mZ.
\]
In particular, we have a weak equivalence
\[
H\mZ\smashover{A}\mathbb S^{0}\simeq H\mZ\otimes\big(\mathbb S^{0}\smashover{A}\mathbb S^{0}\big).
\]
Lemma~\ref{lem:SmashofFrees} allows us to describe \(A\) itself as a suspension spectrum:
\[
A\simeq \Sigma^{\infty}_{+}\Omega\left(\prod_{i\in\mathbb N}\Map^{C_{2}}(G,S^{i\rho_{2}+1})\right).
\]
Lemma~\ref{lem:SmashOverFree} and Corollary~\ref{cor:SmashOverFrees} then gives the relative smash product
\[
\mathbb S^{0}\smashover{A}\mathbb S^{0}\simeq \Sigma^{\infty}_{+}\left(\prod_{i\in\mathbb N}\Map^{C_{2}}(G,S^{i\rho_{2}+1})\right),
\]
which is the desired result.
\end{proof}

The reduction modulo \(2\) is also easily determined from this: we basechange the left factor along the map \(H\mZ\to H\m{\F}_{2}\). Smashing with the mod \(2\) Moore spectrum then immediately gives Theorem~\ref{thm:modTwoAdditiveRelativeSteenrod}.

\subsubsection*{More general quotients}
This same analysis works for the relative homology of more general quotients of the form
    \[
        \Xin/(G\cdot \bar{s}_{i_1},\dots\mid i_j\in I),
    \]
where the elements \(G\cdot\bar{s}_{i_1},\dots\) are algebraically independent elements in \(\pi_{\ast\rho_2}\Xin\) with
    \[
        |s_{i_j}|=i_j\rho_2,
    \] 
as considered in \cite{MR4877605}. We get
    \[
        H\mZ\otimesover{\Xin} \Xin/(G\cdot \bar{s}_{i_1},\dots)\simeq H\mZ\otimes\Sigma^{\infty}_+\left(\prod_{i_j\in I}\Map^{C_2}(G,S^{i_j\rho_2+1})\right).
    \]
This applies in particular to \(\BPG\) and to any of the \(\BPGm\).

Moreover, essential the same analysis will work for any \(\Xin/(G\cdot\bar{t}_{i_1},\dots)\) in space of \(H\mZ\), provided it is a ring. The key step needed was understanding the map from \(A\) factors through \(\mathbb S^0\), which does happen in the ring case. For spectra like \(k^{((G))}[n]\), this is not guaranteed \cite{BHSZ:Extensions}.

\subsection{Splitting the spectrum}
Not only is the smash square of \(H\mZ\) in the category of \(\Xin\)-modules \(H\mZ\) smashed with a suspension spectrum, but also that suspension spectrum splits as a wedge of representation spheres. This gives us complete additive control of the homotopy Mackey functors of the relative smash product. 

The following is just a rewriting of the \(G\)-symmetric monoidalness of the infinite suspension in the case that our space \(X\) is already pointed.
\begin{proposition}\label{prop: Pointed Norms}
    For any pointed \(H\)-space \(X\), we have an equivalence
    \[
        \Sigma^{\infty}_{+} \Map^H(G,X)\simeq N_H^G\big(\Sigma^\infty X\oplus \mathbb S^0\big).
    \]
\end{proposition}

\begin{corollary}\label{cor: Stable Splitting Relative Homology}
    We have a splitting of \(H\mZ\)-modules
    \[
        H\mZ\smashover{\Xin} H\mZ\simeq H\mZ\otimes \bigotimes_{i\geq 1} N_{C_2}^{G}\big(\mathbb S^{i\rho_2+1}\oplus\mathbb S^0\big).
    \]
\end{corollary}
The distributive law for the norm then gives the explicit formula. 

Using the language of \(G\)-indexed sums and products, we have an equivalence
\[
    N_{C_2}^{G}\big(\mathbb S^{i\rho_2+1}\oplus \mathbb S^0\big)\simeq
    \bigoplus_{f\in\Map^{C_2}(G,\{a,b\})} \big(\mathbb S^{i\rho_2+1}\big)^{\otimes f^{-1}(a)}.
\]
For computation, it is helpful to collect the various \(G\)-indexed pieces into a more traditional form. For this, we rewrite using the stabilizer of \(f\).
\begin{definition}
    Let \(H_f\) be the stabilizer of \(f\) in \(G\).
\end{definition}

\begin{proposition}
    The subgroup \(H_f\) is the stabilizer of \(f^{-1}(a)\subseteq G\) and is the stabilizer of \(f^{-1}(b)\subseteq G\). This is the largest subgroup \(H\) for which the decomposition
    \[
        C_{2^n}=f^{-1}(a)\amalg f^{-1}(b)
    \]
    is \(H\)-equivariant.
\end{proposition}

Note that since every subgroup of \(G\) acts freely on it, we have that \(f^{-1}(a)\) and \(f^{-1}(b)\) are automatically free \(H\)-sets. Moreover, since the group is abelian, any \(C_2\)-equivariant map to a set with trivial action is automatically stabilized by \(C_2\). 

These let us identify the indexed tensor.

\begin{proposition}\label{prop: Indexed Tensor}
    For any \(f\in\Map^{C_2}(G,\{a,b\})\), we have an \(H_f\)-equivariant equivalence
    \[
        \big(\mathbb S^{i\rho_2+1}\big)^{\otimes f^{-1}(a)}\simeq \big(N_{C_2}^{H_f}\mathbb S^{i\rho_2+1}\big)^{\otimes |f^{-1}(a)/H_f|}\simeq \mathbb S^{|f^{-1}(a)/H_f|\Ind_{C_2}^{G} (i\rho_2+1)}.
    \]
\end{proposition}

\begin{definition}
    If \(f\in\Map^{C_2}(G,\{a,b\})\) and \(i\geq 1\), then let
    \[
        ||f_i||=\frac{|f^{-1}(a)|}{|H_f|}\Ind_{C_2}^{H_f}(i\rho_2+1).
    \]
\end{definition}

Folding in the additive structure amounts collecting together the various functions in the \(G\)-orbit of an element \(f\).

\begin{corollary}\label{cor: Norm of Exterior Alg}
    We have an equivalence
    \[
        N_{C_2}^{G}\big(\mathbb S^{i\rho_2+1}\oplus\mathbb S^0\big)\simeq
        \bigoplus_{[f]\in\Map^{C_2}(G,\{a,b\})/G} G_+\otimesover{H_f} \mathbb S^{||f_i||}.
    \]
\end{corollary}

Tensoring together all of the pieces then gives us our desired decomposition.

\begin{corollary}
    We have an equivalence of \(H\mZ\)-modules
    \[
        H\mZ\otimesover{\Xin} H\mZ\simeq H\mZ\otimes\bigotimes_{i\geq 1}\left(\bigoplus_{[f_i]\in\Map^{C_2}(G,\{a,b\})/G} G_+\otimesover{H_{f_i}} \mathbb S^{||f_i||} \right).
    \]
\end{corollary}

The distributive law also helps us unpack this part. We first consider a finite number of tensor factors. Here, since the group does not permute the tensor factors, the distributive law can be rewritten as a sum over
\[
    \prod_{i=1}^{n}\Map^{C_2}(G,\{a,b\})\cong \Map^{C_2}(G,\prod_{i=1}^{n}\{a,b\}).
\]
Given a function
\[
    f\in \prod_{i=1}^{n}\Map^{C_2}(G,\{a,b\}),
\]
let 
\[
    f_i=\pi_i\circ f:G\to \{a,b\}
\]
be the projection onto the \(i\)th factor. The distributive law then shows that the summand corresponding to a particular function \(f\) here is given by
\[
    \mathbb S^{||f||}:=\bigotimes_{i=1}^{n} \big(\mathbb S^{i\rho_2+1}\big)^{\otimes f_i^{-1}(a)}.
\]
We again rewrite this using the stabilizers.

\begin{proposition}
    For any \(f\), we have
    \[
        \Stab(f)=\bigcap_{i=1}^{n} \Stab(f_i).
    \]
\end{proposition}
Since \(\Stab(f)\) sits inside all of the \(\Stab(f_i)\), all of the subsets of \(G\) given by \(f_i^{-1}(a)\) are free \(\Stab(f)\)-sets. Identifying the tensor products that show up in the decomposition then follows from just restricting all of the pieces for the individual \(f_i\) to the common subgroup \(\Stab(f)\).

\begin{definition}
    Given a function 
    \[
        f\in\prod_{i=1}^{n}\Map^{C_2}(G,\{a,b\},
    \]
    let
    \[
        ||f||=\sum_{i=1}^{n} \Res_{H_f}^{H_{f_i}} ||f_i||.
    \]
\end{definition}
With this notation, we can combine together all of the pieces
\begin{proposition}
    We have an equivalence
    \[
        \bigotimes_{i=1}^{n} \bigoplus_{[f]\in\Map^{C_2}(G,\{a,b\})} G_+\otimesover{H_f} \mathbb S^{||f_i||}\simeq \\
        \bigoplus_{[f]\in\prod_{i=1}^{n}\Map^{C_2}(G,\{a,b\})} G_+\otimesover{H_f}\mathbb S^{||f||}.
    \]
\end{proposition}
In the discussion, the point \(b\) has played no apparent role. This part of a functions value corresponds to tensor factors with just the zero sphere, so it does not appear in our formulas. When we pass to the colimit over \(n\), this plays a large role, however.
\begin{definition}
    Let
    \[
        \mathcal I={\prod_{i=1}^{\infty}}^w\Map^{C_2}(G,\{a,b\})
    \]
    be the set of all sequences of equivariant functions
    \[
        f=(f_1,\dots)
    \]
    such that for \(i>>0\), \(f_i\) is the constant function at \(b\).
\end{definition}
Note that the stabilizer of a function \(f\in\mathcal I\) only depends on those finitely many coordinate functions which are not constant at \(b\). Moreover,
definition \(||f||\) above extends naturally over \(\mathcal I\), since if \(f_i\) is the constant function at \(b\), then \(f_i^{-1}(a)=\emptyset\).

\begin{corollary}\label{cor: Basis for Relative Homology}
    We have an equivalence of \(H\mZ\)-modules
    \[
        H\mZ\otimesover{\Xin} H\mZ\simeq H\mZ\otimes\bigoplus_{[f]\in\mathcal I/G} G_+\otimesover{H_f} \mathbb S^{||f||}.
    \]
\end{corollary}

\subsection{Homotopy Mackey Functors}
Additively, the \(C_2\)-equivariant homology of
\[
    \mathbb S^{i\rho_2+1}\oplus \mathbb S^0
\]
is
\[
    \pi_{\star}^{C_2}H\mZ\oplus \pi_{\star}^{C_2}H\mZ\cdot (\varsigma\bar{r}_i),
\]
where the degree of \(\varsigma\bar{r}_i\) is
\[
    |\varsigma\bar{r}_i|=i\rho_2+1.
\]
As a \(\pi_{\star}^{C_2}H\mZ\)-module, this is the exterior algebra on a single class \(\varsigma\bar{r}_i\). Building on this, we see that the \(C_2\)-equivariant homology of the norm is given by
\[
    E_{\pi_{\star}^{C_2} H\mZ}(\varsigma\bar{r}_i,\gamma\varsigma\bar{r}_i,\dots,\gamma^{|G|/2-1}\varsigma\bar{r}_i),
\]
where the residual action of \(G=\langle \gamma\rangle\) is given by
\[
    \gamma\cdot\gamma^{j}\varsigma\bar{r}_i=\begin{cases}
        \gamma^{j+1}\varsigma\bar{r}_i & j< |G|/2-1 \\
        -\varsigma\bar{r}_i & j=|G|/2-1,
    \end{cases}
\]
in direct analogy with the algebra \(A\) and its \(C_2\)-equivariant homology. Just as there, we have a natural map
\[
    \Map^{C_2}(G,\{a,b\})\to E_{\pi_{\star}^{C_2} H\mZ}(\varsigma\bar{r}_i,\gamma\varsigma\bar{r}_i,\dots,\gamma^{|G|/2-1}\varsigma\bar{r}_i)/2
\]
given by
\[
    f\mapsto \bigwedge_{g\in f^{-1}(a)/C_2} g\varsigma\bar{r}_i.
\]

\begin{corollary}
We have an isomorphism of \(\m{RO}\)-graded homotopy groups
\[
\pi_{\m{\star}}\big(H\mZ\smashover{\Xin }H\mZ\big)\simeq \bigoplus_{f\in\mathcal I_{n}/G} \Ind_{H_{f}}^{G} \Sigma^{||f||}\pi_{\m{\star}}H\mZ.
\]
\end{corollary}

Sitting inside here is the \(\pi_{\mstar}H\mZ\)-submodule generated by the summands for which \(H_f=G\). These are the only terms which contribute to the geometric fixed points.

\begin{corollary}\label{cor: NonInduced Submodule}
    As an \(\pi_{\mstar}H\mZ\)-module, the submodule generated by the summands corresponding to functions \(f\) with \(H_f=G\) is isomorphic to the exterior algebra
    \[
        E_{\pi_{\mstar}H\mZ}\big(N_{C_2}^{G}(\varsigma\bar{r}_1),\dots\big),
    \]
    where
    \[
        |N_{C_2}^{G}(\varsigma\bar{r}_i)=\Ind_{C_2}^{G}(i\rho_2+1).
    \]
\end{corollary}

Since projective modules are flat, this is a flat module over \(\pi_{\m{\star}}H\mZ\). In particular, Adams'
 original argument shows that we get a Hopf algebroid.
 
\begin{corollary}
The pair \(\Big(\pi_{\m{\star}}H\mZ,\pi_{\m{\star}}\big(H\mZ\smashover{\Xin }H\mZ\big)\Big)\) is a Hopf algebroid in the category of \(\m{RO}\)-graded commutative algebras.
\end{corollary}

We can then use the Baker--Lazarev relative Adams spectral sequence machinery.
\begin{corollary}
For any \(\Xin \)-module \(M\), there is an Adams spectral sequence with \(E_{2}\)-term
\[
\Ext^{s,\m{\star}}_{\Big(\pi_{\m{\star}}H\mZ,\pi_{\m{\star}}\big(H\mZ\smashover{\Xin }H\mZ\big)\Big)}\Big(
\pi_{\m{\star}}H\mZ,\pi_{\m{\star}}\big(H\mZ\smashover{\Xin }M\big)
\Big)
\]
and converging to \(\pi_{\m{\star}-s}(M)\).
\end{corollary}

Of course, to even state this result, we have been using that there is an underlying commutative ring structure. To actually work at all with these terms, we need to describe these products.

\section{The relative dual Steenrod algebra: multiplicative structure}\label{sec:Multiplicative}
\subsection{Operadic Digression}

Since \(H\m{\Z}\) is a commutative ring spectrum and since the Thom reduction 
\[
\Xin\to H\mZ 
\]
is a commutative ring map, the relative smash square is an equivariant commutative ring spectrum. We have described the additive structure, but the method used does not give explicit control over the multiplicative structure. 

For this, we use an equivariant version of a beautiful result of Hahn--Wilson.

\begin{theorem}[Hahn--Wilson]
    The free \(E_1\)-algebra on a regular representation sphere
    \[
        \mathbb S^0[S^{k\rho_2}]
    \]
    admits an \(E_{\rho}\)-multiplication. 
    
    Given an \(E_\rho\)-algebra \(R\),
    the obstructions to extending an \(E_1\)-map out of \(\mathbb S^0[S^{k\rho_2}]\) to an \(E_\rho\)-map lie in the homotopy groups \(\pi_{m\rho_2-1}R\). The augmentation map 
    \[
        \mathbb S^0[S^{k\rho_2}]\to\mathbb S^0
    \]
    is \(E_\rho\).
\end{theorem}

For \(\Xin\) (and any of the regular quotients we normally consider), these homotopy groups are all zero.

\begin{corollary}
    For any element 
    \[
        \bar{r}_i\colon S^{i\rho_2}\to i_{C_2}^{\ast}\Xin,
    \]
    the induced map of associative rings
    \[
        \mathbb S^0[S^{i\rho_2}]\to i_{C_2}^\ast\Xin
    \]
    extends to an \(E_{\rho_2}\)-algebra map.
\end{corollary}

The norm of an \(E_{\rho_2}\)-algebra is automatically an algebra over the coinduced operad
\[
    \Map^{C_2}(G,E_{\rho_2}).
\]
Work of Szczesny shows that a \(\Map^{C_2}(G,E_{\rho_2})\)-algebra have 
\begin{enumerate}
    \item an associative multiplication that is underlying \(E_2\),
    \item norm maps \(C_2\to G\) (from the \(E_\sigma\) factor of \(E_{\rho}\)), and
    \item norm maps \(\{e\}\to G\) (from the \(E_1\) factor of \(E_{\rho}\)).
\end{enumerate}
The quotient map
\[
    \Xin\to\Xin\smashover{A}\mathbb S^0
\]
therefore inherits norms from \(C_2\) to any intermediate group. Put another ways, the classes that look like norms in that they are carried by induced spheres of the approrpriate dimension and come from the norm-like cells in a coinduced sphere are actually norm classes from the underlying \(C_2\)-representation sphere. Making this precise, however, requires us to decompose our operad
\[
    \Map^{C_2}(G,E_{\rho})\simeq E_1\otimes\mathcal O,
\]
and we do not yet have such a decomposition. In light of this, we provide an additional, more direct argument using induction on the subgroup lattice.

\subsection{Products of induced classes}
By the Frobenius relation, understanding products of any classes with a proper stabilizer \(H\) is equivalent to understanding the \(H\)-equivariant products. This necessitates understanding better the restrictions in topology. For this subsection, let \(K=C_{2^m}\) be a subgroup of \(G\) containing \(C_2\).

\begin{proposition}
We have an equivalence of \(K\)-\(E_{\infty}\)-ring spectra
\[
i_{K}^{\ast}\big( H\mZ\smashover{\Xin } H\mZ\big)\simeq 
(H\mZ\smashover{\Xi_{m}}H\mZ)\otimes \Sigma^{\infty}_+\left(\prod_{i=1}^{2^{n-m}-1} \Map^{C_{2}}(K,B\BUR)\right).
\]
\end{proposition}
\begin{proof}
    Since restriction is strong symmetric monoidal, and since
    \[
        i_K^\ast H\mZ=H\mZ,
    \]
    we have an equivalence of \(K\)-\(E_\infty\) ring spectra
    \[
        i_K^\ast\big( H\mZ\smashover{\Xin} H\mZ\big) \simeq H\mZ\smashover{i_K^\ast \Xin} H\mZ.
    \]
    By the Thom diagonal, we have an equivalence of \(K\)-\(E_\infty\) ring spectra
    \[
        i_K^\ast\Xin\simeq \Xi_m\otimes\bigotimes_{i=1}^{2^{n-m}-1} \Xi_m\simeq\Xi_m\otimes\Sigma^{\infty}_{+} \left(\prod_{i=1}^{2^{n-m}-1} \Map^{C_2}(K,\BUR)\right).
    \]
    The map
    \[
        \bigotimes_{i=1}^{2^{n-m}-1}\Sigma^{\infty}_+ \Map^{C_2}(K,\BUR)\to i_K^\ast\Xin\to H\mZ
    \]
    factors through the augmentation to \(\mathbb S^0\), since all summands but the first are slice positive. This means we again have an equivalence
    \[
        H\mZ\smashover{i_K^\ast \Xin} H\mZ\simeq \big(H\mZ\smashover{\Xi_m} H\mZ\big)\otimes \Sigma^{\infty}_{+}\left(\prod_{i=1}^{2^{n-m}-1} \Map^{C_2}(K,B\BUR)\right),
    \]
    where we have use the equivalence of \(K\)-\(E_\infty\)-spaces
    \[
        B\prod\Map^{C_2}(K,\BUR)\simeq\prod\Map^{C_2}(G,B\BUR).\qedhere
    \]
\end{proof}

By our additive computation, the homotopy of \(H\mZ\smashover{\Xi_{m}} H\mZ\) is free over that of \(H\mZ\), and hence the K\"unneth spectral sequence collapses with no extensions.

\begin{corollary}
We have an isomorphism of \(\m{RO}\)-graded commutative algebras
\[
\pi_{\m{\star}} \big(i_{K}^{\ast} H\mZ\smashover{\Xin } H\mZ\big)\cong
\pi_{\m{\star}}(H\mZ\smashover{\Xi_{m}}H\mZ)\Boxover{\pi_{\m{\star}}H\mZ} H_{\m{\star}}\left(\prod_{i=1}^{2^{n-m}-1} \Map^{C_{2}}(K,B\BUR);\mZ\right).
\]
\end{corollary}

Since the homology of \(B\BUR\) is free, so too is the homology of the coinduced \(B\BUR\):
\begin{theorem}[{\cite[Theorem 5.9]{HillFreeness}}]
    We have an isomorphism of \(\mRO\)-graded Tambara functors
    \[
        H_{\m{\star}}\big(\Map^{C_{2}}(C_{2^{m}},B\BUR);\mZ\big)\cong N_{C_2}^{K} \pi_{\mstar} H\mZ [\bar{y}_1,\dots]/(\bar{y}_i^2-a_\sigma\bar{y}_{2i+1}),
    \]
    where \(|\bar{y}_i|=i\rho_2+1\).
\end{theorem}

Note that the degree of the class \(\bar{y}_i\) is exactly that of the class \(\varsigma\bar{r}_i\). We believe this is not coincidental (see Conjecture~\ref{conj: Thom Version} below), and that we can choose the \(\varsigma\bar{r}_i\) to essentially be these, using the \(\bar{t}_i\) generators for \(i_{C_2}^\ast\Xin\) instead of the usual \(\bar{r}_i\) \cite{BHSZ:ETheory}.

Finally, since the \(\m{RO}\)-graded homology of \(\Map^{C_{2}}(C_{2^{m}},B\BUR)\) is again free, we can express the homology of the product of copies as the box product.

\begin{corollary}
We have an isomorphism of \(\m{RO}\)-graded commutative algebras
\[
\pi_{\m{\star}} \big(i_{C_{2^{m}}}^{\ast} H\mZ\smashover{\Xin } H\mZ\big)\cong
\pi_{\m{\star}}(H\mZ\smashover{\Xi_{m}}H\mZ)\Boxover{\pi_{\m{\star}}H\mZ} H_{\m{\star}}\big(\Map^{C_{2}}(C_{2^{m}},B\BUR);\mZ\big)^{\Box (2^{n-m}-1)}.
\]
\end{corollary}

\subsection{Non-induced classes}
By Corollary~\ref{cor: NonInduced Submodule}, the sub-module of non-induced classes looks like an exterior algebra on the classes
\[  
    N_{C_2}^{G}(\varsigma\bar{b}_i)\colon S^{\Ind_{C_2}^{G}(i\rho_2+1)}\to H\mZ\otimesover{\Xin}H\mZ.
\]
Our key step here is decomposing the geometric fixed points. 

\begin{notation}
    Let 
    \[
        \Lambda_i=\Ind_{C_2}^{G}(i\rho_2+1)-(i+1)
    \]
    be the non-fixed summand.
\end{notation}

\begin{proposition}
    Additive, as an \(\F_2[b]\)-module, the homotopy of the geometic fixed points of \(H\mZ\otimesover{\Xin}H\mZ\) are given by
    \[
        \F_2[b]\otimes E_{\F_2}(x_2, x_3,\dots),
    \]
    where 
    \[
        x_{i+1}=a_{\Lambda_i} N_{C_2}^{G}\varsigma\bar{r}_i
    \]
    has degree \(i+1\).
\end{proposition}

We also have a more general method for finding the geometric fixed points, since they play nicely with the norm and tensor.

\begin{proposition}
    The geometric fixed points are given by
    \[
        \Phi^{G}\big(H\mZ\otimesover{\Xin} H\mZ\big)\simeq H\F_2[b]\otimesover{MO}H\F_2[b'],
    \]
    where \(b,b'\) are in degree \(2\).
\end{proposition}

Our decomposition of \(H\mZ\otimesover{\Xin}H\mZ\) is asymmetrical in the two copies of \(H\mZ\), since it arises from an analysis
\[
    H\mZ\otimes\big(\mathbb S^0\otimesover{A} \mathbb S^0\big).
\]
The leftmost copy of \(b\) comes from the geometric fixed points of the left copy of \(H\mZ\), and the whole analysis is in \(H\mZ\)-modules. We may therefore basechange down to \(H\F_2\) here. 

When we take the geometric fixed points of tha associative algebra \(A\), we again get a tensor product of free associative algebras, but now on simpler classes.

\begin{proposition}
    As associative rings, the geometric fixed points of \(A\) are given by 
    \[
        \bigotimes_{i=1}^{\infty}\mathbb S^0[S^i]\to \Phi^{G}A,
    \]
    where the tensor factor corresponding to \(i\) goes in as 
    \[
        f_i:=a_{\bar{\rho}_G}^i N_{C_2}^{G}\bar{r}_i.
    \]
\end{proposition}

The map to 
\[
    \pi_\ast MO\cong \F_2[x_2,x_4,\dots]
\]
was computed in \cite{HHR}: 
\[
    f_i\mapsto\begin{cases}
        x_i & i\neq 2^j-1 \\
        0 & i=2^j-1.
    \end{cases}
\]
Moreover, the relative tensor over the subalgebra generated by 
\[
    f_1, f_3, f_7,\dots, f_{2^j-1},\dots
\]
was shown to be a polynomial algebra on a \(2\)-dimensional class: \(b'\).

It suffices to understand the homotopy of 
\[
    H\F_2\otimesover{MO}H\F_2.
\]

\begin{proposition}
    The homotopy groups of 
    \[
        H\F_2\otimesover{MO}H\F_2
    \]
    are a polynomial algebra on classes in odd degrees:
    \[
        \varsigma x_2, \varsigma x_4,\dots
    \]
\end{proposition}
\begin{proof}
    The Thom isomorphism gives a change-of-rings isomorphism
    \[
        H\F_2\otimesover{MO}H\F_2\simeq (H\F_2\otimes H\F_2)\otimesover{H\F_2\otimes BO_+}H\F_2.
    \]
    The K\"unneth spectral sequence here collapses, giving an associated graded of the form
    \[
        E_{\F_2}(\varsigma x_i\mid i\neq 2^j-1).
    \]
    Kochman's computation of the Dyer--Lashof action on the homology of \(BO\) \cite{Kochman} shows that we have
    \[
        Q_1 x_i\equiv x_{2i+1} \mod \text{decomposables.}
    \]
    Since the Dyer--Lashof generators commute with the homology suspension, we have
    \[
        Q_0(\varsigma x_i)=\varsigma(Q_1x_i)=\varsigma x_{2i+1},
    \]
    giving the desired result.
\end{proof}

Tracing the induced map on K\"unneth spectral sequences, we see the following.

\begin{corollary}
    Modulo decomposables and \(a_\lambda\)-torsion elements, we have
    \[
        N_{C_2}^{G}(\varsigma\bar{b}_i)^2=a_{\Ind_{C_2}^{G}\varsigma} N_{C_2}^{G}(\varsigma\bar{b}_{2i+1}).
    \]
\end{corollary}
\begin{proof}
    This is the only option for degree reasons. 
\end{proof}

\begin{remark}
    The representation \(\Ind_{C_2}^{G}\sigma\) is the sum of all of the irreducible representations of \(G\) which are also free. These correspond to the different choices of a pair of a primitive \(2^n\)th root of unity and its conjugate, and so all give forms of \(\lambda\).
\end{remark}

\begin{remark}
    The multiplicative relations show that we cannot kill off a single \(\bar{r}_i\) from \(\Xin\) and get a ring. Instead, we have to kill off an entire infinite family to wipe out the entire polynomial subalgebra corresponding to these. This is even true for \(\MUR\). This is analogous to the \cite{BHLSZ:KTheory}, which showed that the \(\BPGm\) are essentially the only quotients of \(\BPG\) of that form which even have a chance to be rings.
\end{remark}

By induction, we see that for all \(K\subseteq G\), we have
\[
    \Phi^{K}\big(H\mZ\otimesover{\Xin}H\mZ\big)\simeq \Phi^{K}\big(H\mZ\otimesover{\Xi_m}H\mZ\big)\otimes\Sigma^{\infty}_{+}\prod_{i=1}^{2^{n-m}-1} BBO.
\]
The computation above suggests that we have an extra copy of \(BBO\) in the geometric fixed points, which would symmetrize this into
\[
    H\F_2[b]\otimes \Sigma^{\infty}_+ \prod_{i=1}^{2^{n-m}} BBO,
\]
which is equivalently
\[
    \Phi^K\left(H\mZ\otimes\Sigma^{\infty}_+\Map^{C_2}(G,B\BUR)\right).
\]

\begin{conjecture}\label{conj: Thom Version}
    We have an equivalence of \(\Map^{C_2}(G,E_{2+\sigma})\)-algebras
    \[
        H\mZ\smashover{\Xin} H\mZ\simeq H\mZ\otimes \Sigma^{\infty}_+\Map^{C_2}(G,B\BUR).
    \]
\end{conjecture}

\section{The absolute case}\label{sec:AbsoluteCase}

The relative homology fits into a pushout square of equivariant commutative ring spectra
\[
    \begin{tikzcd}
        {H\mF_2\otimes \Xin}
            \ar[r]
            \ar[d]
            &
        {H\mF_2\otimes H\mZ}
            \ar[d]
            \\
        {H\mF_2}
            \ar[r]
            &
        {H\mF_2\otimesover{\Xin}H\mZ.}
    \end{tikzcd}
\]
We have complete computational control over all but the top right term in this square. The Greenlees--Serre spectral sequence provides a way to leverage these to get the fourth.

\subsection{An equivariant Greenlees--Serre spectral sequence}

\begin{theorem}
    Let \(f\colon A\to R\) be a map of equivariant commutative ring spectra with \(A\) connective, let \(p\colon A\to P^0(A)\) be the zeroth slice section map, and let \(Q\) be the equivariant commutative ring spectrum defined by the pushout square
    \[
        \begin{tikzcd}
            {A} 
                \ar[r, "f"]
                \ar[d, "p"']
                &
            {R}
                \ar[d]
                \\
            {P^0(A)}
                \ar[r]
                &
            {Q.}
        \end{tikzcd}
    \]
    Let \(Gr(A)\) be the slice associated graded of \(A\), viewed as a commutative \(P^0(A)\)-algebra. Then we have a spectral sequence of graded Tambara functors with
    \[
        E_2^{(s,\m{t})}=\pi_{\m{t}-s}\big(Q\smashover{P^0(A)} Gr_{\dim \m{t}}(A)\big)\Rightarrow \pi_{\m{t}-s}(R).
    \]
    The differentials satisfy
    \[
        |d_r|=(-1,r).
    \]
    
    When \(R\) is also \(0\)-connective, this converges strongly.
\end{theorem}
\begin{proof}
    Consider the slice filtration of \(A\), \(P_k(A)\), given by the slice covers of \(A\). Since \(A\) is \(0\)-connective, the natural map to the slice truncation
    \[
       A\to P^k A 
    \]
    is always a map of \(G\)-commutative \(A\)-algebras, and hence the filtration \(P_k(A)\) is actually a filtration in \(A\)-modules. This filtration is Hausdorff, in the sense that
    \[
        \lim_{\longleftarrow} P_k(A)\simeq 0,
    \]
    and of course, 
    \[
        A=P_0A.
    \]
    
    This gives a filtration of 
    \[
    R\simeq R\smashover{A} P_0(A),
    \]
    by taking 
    \[
    F_k(R)= R\smashover{A}P_k(A).
    \]
    The \(k\)th associated graded piece is
    \[
    R\smashover{A} P^k_k(A).
    \]
    By construction, the \(A\)-action on \(P^k_k(A)\) factors through \(P^0(A)\). In other words, we have an equivalence
    \[
    R\smashover{A} P^k_k(A)\simeq R\smashover{A} P^0(A)\smashover{P^0(A)} P^k_k(A).
    \]
    This means that the associated graded for \(R\) can be rewritten as 
    \[
    \big(R\smashover{A} P^0(A)\big)\smashover{P^0(A)} P^\ast_\ast(A)= Q\smashover{P^0(A)} P^\ast_\ast(A).
    \]
    
    Since the slice filtration of a \(G\)-commutative monoid is naturally a \(G\)-commutative monoid in filtered \(G\)-spectra, and since all maps considered are \(G\)-commutative ring maps, the induced filtration on \(R\) is a \(G\)-commutative monoid in filtered spectra. The spectral sequence is therefore a spectral sequence of graded Tambara functors.
    
    For convergence, we start with the observation that the connectivity of \(P_kA\) goes to infinity with \(k\) (at worst as about \(k/|G|\)). If \(R\) is connective, then the K\"unneth spectral sequence shows that the connectivity of \(R\otimesover{A} P_kA\) also goes to infinity, or equivalent, for any \(n\) and for all \(k>\!>0\), 
    \[
        \m\pi_n\big(R\otimesover{A}P_kA\big)=0.
    \]
    This gives strong convergence.
\end{proof}
    
There is a K\"unneth spectral sequence that allows us to compute the \(E_2\)-term of this Greenlees--Serre spectral sequence, due to Lewis--Mandell \cite{LewisMandell}. We have
\[
E_2^{V,s}=\Tor^{-s}_{\pi_\mstar P^0(A)}\big(\pi_\mstar Q,\pi_\mstar Gr_k(A)\big)_{V}\Rightarrow \pi_{V-s}\big(Q\smashover{P^0(A)} Gr_k(A)\big).
\]
In our case, however, things are greatly simplified. Since \(H\mF_2\) is \(\Xin\)-oriented, Corollary~\ref{cor:Homotopy of EXin} shows that we have an equivalence of \(H\mF_2\)-associative algebras
\[
    H\mF_2\otimes A_n^h\simeq H\F_2\otimes\Xin.
\]
This is a wedge of slice spheres tensored with \(H\mF_2\), and hence is a wedge of slices, so it is its own slice associated graded.

\begin{proposition}
    The zero slice of \(H\mF_2\otimes \Xin\) is \(H\mF_2\), and the slice associated graded is
    \[
    H\mF_2\otimes A^h_n,
    \]
    where the grading on \(A^h_n\) is the usual monomial grading.
\end{proposition}

For us, the key fact is that this is a free \(H\mF_2\)-module. In particular, smashing this over \(H\mF_2\) is easy to understand.

\begin{proposition}
    For any \(0\)-connective \((H\mF_2\otimes\Xin)\)-algebra \(R\), the associated graded for \(R\) is
    \[
    \big(R\smashover{H\mF_2\otimes\Xin}H\mF_2\big)\otimes A^h_n.
    \]
\end{proposition}

\begin{corollary}
    The Greenlees--Serre associated graded for \(H\mF_2\otimes H\mZ\) is
    \[
    \pi_{\mstar}\big(H\mF_2\smashover{\Xin} H\mZ\big)[G\cdot \bar{m}_1,\dots].
    \]
\end{corollary}

\subsection{Proof of concept: the \texorpdfstring{\(C_2\)}{C-two} Steenrod Algebra}
Using the Greenlees--Serre spectral sequence, we can give a new approach to the \(C_2\)-equivariant dual Steenrod algebra, complementing the foundational work of Hu--Kriz \cite{HuKriz}. We begin with some easily computational corollaries from our previous computations. We work here over \(\BPR\) to simplify our computations; work of Roytman showed this is \(E_{2\rho}\).

First note that for \(C_2\), our decomposition underlying the relatie smash product is
\[
H\mF_2\smashover{\BPR} H\mZ\simeq H\mF_2\otimes\Sigma^{\infty}_{+}\left(\prod_{i=1}^{\infty} S^{(2^i-1)\rho_2+1}\right).
\]
When we decompose the Cartesian product, this is just a wedge of ordinary \(C_2\)-representation spheres. Put another way, a basis for the \(RO\)-graded homology occurs all in level \(C_2\), so we may without loss of generality work instead with an \(RO(C_2)\)-grading. 

\begin{proposition}
    The \(\BPR\)-relative homology \(H\mF_2\smashover{\BPR} H\mZ\) is
    \[
        \mathbb M_2[\bar{\tau}_1,\dots]/(\bar{\tau}_i^2-a_\sigma\bar{\tau}_{i+1},\forall i\geq 1),
    \]
    where the degree of \(\bar{\tau}_i\) is \(\big((2^i-1)\rho_2+1\big)\).
\end{proposition}
\begin{proof}
    The only thing to check is the multiplicative relation. Passing to geometric fixed points, we have as a ring
    \[
        \Phi^{C_2}\big(H\mF_2\smashover{\BPR} H\mZ\big)\simeq H\F_2[b]\smashover{H\F_2} H\F_2[{b'}^2],
    \]
    where \(b\) and \(b'\) have degree \(1\). The geometric fixed points of the class \(\bar{\tau}_i\) is in degree \(2^i\), so we see that all of the classes \(\bar{\tau}_i\) must have multiplicative relations that rebuild a polynomial ring. By considering the bidegrees, the only possible relations are that the ring structure is as listed.
\end{proof}

There is also a multiplicative relation when we reduce modulo \(2\). The proof is essentially identical.

\begin{proposition}
    The \(\BPR\)-relative homology \(H\mF_2\smashover{\BPR} H\mF_2\) is
    \[
        \mathbb M_2[\bar{\tau}_0,\dots]/(\bar{\tau}_i^2-a_\sigma\bar{\tau}_{i+1},\forall i\geq 0),
    \]
    where the degree of \(\bar{\tau}_i\) is \(\big((2^i-1)\rho_2+1\big)\).
\end{proposition}

Similarly, the homology of \(\BPR\) is
\[
    H_\star \BPR\cong \mathbb M_2[\bar{\xi}_1,\dots],
\]
where \(\bar{\xi}_i\) is in degree \((2^i-1)\rho_2\). This gives us the needed input for the Greenlees spectral sequence.

\begin{theorem}
    The Greenlees--Serre spectral sequence computing \(H\mF_2\otimes H\mZ\) collapses at 
    \[
        E_2=\mathbb M_2[\bar{\xi}_1,\dots]\otimesover{\mathbb M_2}\mathbb M_2[\bar{\tau}_1,\dots]/\bar{\tau}_i^2-a_{\sigma}\bar{\tau}_{i+1},
    \]
    where
    \[
        \mathbb M_2=\pi_{\star}H\F_2
    \]
    is the \(RO(C_2)\)-graded homology of a point. There are no additive extensions.
\end{theorem}
\begin{proof}
    The classes \(\bar{\xi}_i\) all come from the homotopy of \(H\mF_2\otimes\BPR\), and hence are necessarily permanent cycles. Since we are computing Bredon homology, everything in \(\mathbb M_2\) is also a permanent cycle. The only classes which might not be permanent cycles are the classes \(\bar{\tau}_i\) for \(i\geq 1\), so we must show only that these all survive the spectral sequence. We do this by induction on \(i\).
    
    The base case, \(i=1\), is handled by noting that the Hurewicz image of \(\bar{v}_1\) is zero. For degree reasons, the only possible target of the differential on \(\bar{\tau}_1\) is a multiple of \(\bar{\xi_1}\), but we deduce this must be zero.
    
    Assume by induction that all \(\bar{\tau}_i\) for \(i<k\) are permanent cycles. This means that \(\bar{\tau}_{k-1}^2=a_\sigma\bar{\tau}_{k}\) is a permanent cycle. However, up until this point, the spectral sequence is a sum of free \(\mathbb M_2\)-modules and the possible targets of any differential on \(\bar{\tau}_k\) are \(a_\sigma\)-torsion free, we deduce that \(\bar{\tau}_k\) must also be a permanent cycle. 
    
    Since \(E_2=E_\infty\) is a free module over \(\mathbb M_2\), there are no additive extensions.
\end{proof}

As Hu--Kriz show, there are non-trivial multiplicative extensions.

\bibliographystyle{plain}

\bibliography{RelHZ}

\end{document}